\documentclass[a4paper,12pt]{scrartcl}
\usepackage{amsmath,amssymb,amsthm}
\usepackage{hyperref}
\usepackage[english]{babel}
\usepackage{amsbsy}
\usepackage{amstext}
\usepackage[titletoc,title]{appendix}
\newcommand{\ct}{c(\theta)}
\newcommand{\ctg}{c(\theta,\gamma)}

\newcommand\myeq{\mathrel{\overset{\makebox[0pt]{\mbox{\normalfont\tiny\sffamily def}}}{=}}}

\newtheorem{thm}{Theorem}[section]
\newtheorem{prop}[thm]{Proposition}
\newtheorem{lemma}[thm]{Lemma}
\newtheorem*{lemma*}{Lemma}

\theoremstyle{definition}

\theoremstyle{remark}
\newtheorem{remark}[thm]{Remark}

\newcommand{\rank}{\mathop{\textrm{rank}}}
\newcommand{\dist}{\mathop{\textrm{dist}}}

\begin{document}

\global\long\def\bbr{\mathbb{R}}
\global\long\def\bbz{\mathbb{Z}}
\global\long\def\bbn{\mathbb{N}}
\global\long\def\bbq{\mathbb{Q}}
\global\long\def\calc{\mathcal{C}}
\global\long\def\calb{\mathcal{B}}
\global\long\def\cali{\mathcal{I}}
\global\long\def\calj{\mathcal{J}}
\global\long\def\sep{\;:\;}
\global\long\def\sub{\subseteq}
\global\long\def\eps{\varepsilon}
\global\long\def\mto{\xrightarrow[m\to\infty]{}}
\global\long\def\lto{\xrightarrow[\ell\to\infty]{}}
\global\long\def\hto{\xrightarrow[h\to\infty]{}}
\global\long\def\bt{\mathbf{t}}
\global\long\def\be{\mathbf{e}}
\global\long\def\bg{\mathbf{g}}
\global\long\def\bw{\mathbf{w}}
\global\long\def\bb{\mathbf{b}}
\global\long\def\bp{\mathbf{p}}
\global\long\def\bq{\mathbf{q}}
\global\long\def\bm{\mathbf{m}}
\global\long\def\bn{\mathbf{n}}
\global\long\def\br{\mathbf{r}}
\global\long\def\bv{\mathbf{v}}
\global\long\def\bu{\mathbf{u}}
\global\long\def\bx{\mathbf{x}}
\global\long\def\by{\mathbf{y}}
\global\long\def\bz{\mathbf{z}}
\global\long\def\bi{\mathbf{i}}
\global\long\def\bh{\mathbf{h}}
\global\long\def\bg{\mathbf{g}}
\global\long\def\btheta{\boldsymbol{\theta}}
\global\long\def\bphi{\boldsymbol{\varphi}}
\global\long\def\bgamma{\boldsymbol{\gamma}}
\global\long\def\bell{\boldsymbol{\ell}}
\global\long\def\bzero{\mathbf{0}}
\global\long\def\bbf{\mathbf{\mathbb{F}}}
\global\long\def\RR{\mathbf{\mathbb{R}}}
\global\long\def\QQ{\mathbf{\mathbb{Q}}}
\global\long\def\ZZ{\mathbf{\mathbb{Z}}}
\global\long\def\bbfq{\mathbf{\mathbb{F}}_{q}}
\global\long\def\Z{\bbfq\left[t\right]}
\global\long\def\Q{\bbfq\left(t\right)}
\global\long\def\R{\bbfq\left(\left(\frac{1}{t}\right)\right)}

\newcommand{\cht}{c_{\bg}\left(\btheta\right)} 
\newcommand{\chtg}{c_{\bg}\left(\btheta,\bgamma\right)} 
\newcommand{\ch}{c_{\bg}}

\title{Solution of Cassels' Problem on a Diophantine Constant over Function Fields}	





	\author{Efrat Bank
		\thanks{Department of Mathematics, University of Michigan, Ann Arbor MI, 48109, USA, \href{mailto:ebank@umich.edu}{ebank@umich.edu}}
		\and  Erez Nesharim \thanks{School of Mathematical Sciences, Tel Aviv University, Tel Aviv 69978, Israel, \href{mailto:ereznesharim@post.tau.ac.il}{ereznesharim@post.tau.ac.il}}
		\and Steffen H\o jris Pedersen  \thanks{Department of Mathematics, Aarhus University, 8000 Aarhus C, Denmark, \href{mailto:steffenh@math.au.dk}{steffenh@math.au.dk}}
	}

\date{\today}

\maketitle
\begin{abstract}
This paper deals with an analogue of Cassels' problem on inhomogeneous Diophantine approximation in function fields. The inhomogeneous approximation constant of a Laurent series $\theta\in\R$ with respect to $\gamma\in\R$ is defined to be $c(\theta,\gamma)=\inf_{0\neq N\in\Z}|N|\cdot|\langle N\theta - \gamma \rangle|$. We show that \sloppy $\inf_{\theta\in\R} \sup_{\gamma\in \R} c(\theta,\gamma) = q^{-2}$, and prove that for every $\theta$ the set $BA_{\theta}=\left\{\gamma\in\R\sep c(\theta,\gamma)>0\right\}$ has full Hausdorff dimension. Our methods generalize easily to the case of vectors in $\R^d$.
\end{abstract}

\section{Introduction}
For a real number $\theta$, denote by $\left\langle \theta\right\rangle =\theta-\left\lfloor \theta+\frac{1}{2}\right\rfloor$
the representative in $\left[-\frac{1}{2},\frac{1}{2}\right)$ of
$\theta$ modulo the integers, and let $\left|\theta\right|$ denote
the absolute value of $\theta$. In these notation, $\left|\left\langle \theta\right\rangle \right|$ is the distance from $\theta$ to the integers.

A main topic in Diophantine approximation deals with the inhomogeneous approximations of a real number (see \cite{C57}). Given two real numbers $\theta$ and $\gamma$, define the \textit{inhomogeneous approximation constant} of $\theta$ with respect to $\gamma$ as
\begin{equation}\label{eq:casselsConstantInf}
\ctg\myeq\inf_{n\neq 0}|n|\cdot|\langle n\theta - \gamma \rangle|.
\end{equation}
Also define the set
\begin{equation}\label{eq:BAtheta}
BA_{\theta} \myeq \left\{ \gamma\in\bbr\sep c\left(\theta,\gamma\right)>0 \right\}.
\end{equation}
It was proved by \cite{BM} (cf. \cite{Kim} for a second proof):
\begin{thm}\label{thm:Kim}
For every $\theta\in\bbr\setminus\bbq$, the set $BA_{\theta}$ has zero Lebesgue measure.
\end{thm}
On the other hand, the following result concerning $BA_{\theta}$ is proved in \cite{T} (see also \cite[Theorem 2.3]{BM}):
\begin{thm}\label{thm:Tseng}
For every $\theta\in\bbr$, the set $BA_{\theta}$ has Hausdorff dimension $1$.
\end{thm}

We mention that subsets of $\RR^d$ with positive Hausdorff dimension are uncountable, and that subsets with positive Lebesgue measure in $\bbr^d$ have maximal dimension, i.e., $d$ (see \cite{Falconer} for the definition of Hausdorff dimension). In view of that, Theorem~\ref{thm:Kim} states that the set $BA_{\theta}$ is small, while Theorem~\ref{thm:Tseng} states that $BA_{\theta}$ is large, and in particular, not empty. Therefore, for every $\theta$ there exists a $\gamma$ such that $\ctg>0$. This leads to the definition of the following two constants:
\begin{equation}\label{eq:casselsConstant}
\ct\myeq\sup_{\gamma}\ctg,
\end{equation}
and
\begin{equation}\label{eq:casselsConstantTrue}
c \myeq \inf_{\theta}c(\theta).
\end{equation}
Khinchine \cite{Kh} proved that $c>0$. Davenport \cite{D} was the first to give an explicit lower bound on $c$. The problem of finding the exact value of it was posed by Cassels \cite[p.86]{C57}. According to \cite{M}, the best estimate of $c$ was found in \cite{G}:
\begin{thm}\label{thm:Godwin}
\[
0.14078 < c < 0.2114\,.
\]
\end{thm}

In this work we study the analogues of these constants in the context of function fields.

\begin{remark}
Some authors consider a constant which is similar to the one defined in \eqref{eq:casselsConstantTrue}:
\begin{equation}\label{eq:casselsConstantLimInf}
\tilde{c}\myeq\inf_{\theta}\sup_{\gamma}\liminf_{n\to\infty}|n|\cdot|\langle n\theta - \gamma \rangle|.
\end{equation}
By definition we have $c\leq \tilde{c}$, and we are not aware of any result regarding equality. However, the function fields analogues of those constants coincide (cf. Theorem \ref{thm:liminf}).
\end{remark}

\subsection{Higher Dimensions}\label{subsec:higherDim}
Throughout the paper, we will denote vectors by bold symbols, and their coordinates with superscripts. Assume $d\geq1$. A \textit{weight} in $\bbr^d$ is a vector $\br\in\bbr^d$ with $r^1+\cdots+r^d=1$, $r^s \geq 0$ for any $1\leq s \leq d$. Given a weight $\br$ and $\btheta, \bgamma\in \bbr^d$, define the \textit{approximation constant} with weight $\br$ of $\btheta$ with respect to $\bgamma$ by
\[
c_{\br}(\btheta,\bgamma)\myeq \inf_{n\neq0}\max_{1\leq s\leq d}\left(\left|n\right|^{r^s}\left|\left\langle n\theta^{s}-\gamma^{s}\right\rangle \right| \right),
\]
and let
\[
BA_{\btheta}(\br) \myeq \left\{ \bgamma\in\bbr^d\sep c_{\br}\left(\btheta,\bgamma\right)>0 \right\}.
\]
As in the one dimensional case, define
\[
c_{\br}\left(\btheta\right)\myeq\sup_{\bgamma}c_{\br}(\btheta,\bgamma),
\]
and
\begin{equation}\label{eqn:casselsConstantHigherDimension}
c_{\br} \myeq \inf_{\btheta}c_{\br}\left(\btheta\right).
\end{equation}
A higher dimensional version of Theorem \ref{thm:Kim} is proved in \cite{S12} by dynamical methods:
\begin{thm}\label{thm:Shapira}
For almost every $\btheta\in\bbr^d$ (described explicitly), the set $BA_{\btheta}\left(\frac{1}{d},\ldots,\frac{1}{d}\right)$ has measure zero.
\end{thm}
The higher dimensional version of Theorem \ref{thm:Tseng} appeared only recently in \cite{PM}, extending a result proved independently by \cite{BHKV} and \cite{ET} about the weight $\br=\left(\frac{1}{d},\ldots,\frac{1}{d}\right)$:
\begin{thm}\label{thm:PM}
For every weight $\br$ and $\btheta\in\bbr^d$, the set $BA_{\btheta}\left(\br\right)$ has dimension $d$.
\end{thm}
As for \eqref{eqn:casselsConstantHigherDimension}; Cassels \cite[Theorem $X$]{C57} showed that $c_{\left(\frac{1}{d},\ldots,\frac{1}{d}\right)} > 0$, and an explicit lower bound was established in \cite{BL}:
\begin{thm}\label{thm:BL}
For every $d\geq1$
\[
c_{\left(\frac{1}{d},\ldots,\frac{1}{d}\right)}\geq\frac{1}{72d^28^\frac{1}{d}}.
\]
\end{thm}
We know of no results regarding $c_{\br}$ for a general weight $\mathbf{r}$.
\subsection{The Function Fields Analogue of Diophantine Approximation}\label{subsec:functionFieldAnalog}
The function fields analogue of Diophantine approximation has been studied since the work of Artin \cite{Artin}. It is sometimes referred to as Diophantine approximation in positive characteristic. Every statement in Diophantine approximation has an analogous statement in this context. Let us introduce the dictionary which is used for translating statements (and sometimes, their proofs) from one context to the other. Let $q$ be a prime power, and let $\bbf_q$ be the field with $q$ elements. Define an absolute value on $\Z$ by $|N|\myeq q^{\deg N}$ for $0\neq N\in\Z$, and $|0|=0$. Extend this definition to the fraction field, the field of rational functions $\Q$, by $\left\lvert\frac{M}{N}\right\rvert\myeq q^{\deg M-\deg N}$ where $M, N\in\Z$, $N\neq0$. The field $\R$ of formal Laurent series in $t$ with finite number of non zero coefficients of positive powers of $t$, is the completion of $\Q$ with respect to this absolute value. Extending the absolute value continuously to $\R$, gives that the absolute value of a non zero $\theta \in \R$, written as
\[
\theta=\sum_{i=-\deg\theta}^{\infty}\theta_{i}t^{-i},
\]
where $\deg\theta\myeq\max\left\{ -i\sep\theta_{i}\neq0\right\}$, is
\[
\left|\theta\right|=q^{\deg\theta}.
\]

The set
\[
I\myeq\left\{\theta \in \R \sep |\theta| < 1 \right\}.
\]
is a natural set of representatives for elements in $\R$ up to the equivalence relation of having a difference which is a polynomial. We denote $\left\langle \theta\right\rangle =\sum_{i=1}^{\infty}\theta_{i}t^{-i}$ and consider it to be the representative of $\theta$ in $I$. We call $\left\langle \theta\right\rangle$ and $\theta - \left\langle \theta\right\rangle$ the \textit{fractional part} and the \textit{polynomial part} of $\theta$, respectively. These definitions give the dictionary:
\[
\begin{array}{lll}
\bbf_q[t]   & \leftrightsquigarrow & \ZZ\\
\bbf_q(t)   & \leftrightsquigarrow & \QQ\\
\R          & \leftrightsquigarrow & \RR\\
|\theta|=q^{\deg\theta}                             & \leftrightsquigarrow & |\theta|\\
|\langle\theta\rangle|=\dist \left(\theta,\Z\right) & \leftrightsquigarrow & |\langle\theta\rangle|=\dist \left(\theta,\bbz\right).\\
\end{array}
\]

\subsection{Previous Works in Inhomogeneous Approximation in Function Fields}\label{subsec:inhomogeneousFunctionField}
The analogue of inhomogeneous approximation in function fields was studied in \cite{Mahler}. Recently, this subject has regained interest, parallel to a significant progress in the real case \cite{K11,KN,CF,FK}. Let us use the dictionary described above in order to define the function fields analogues of \eqref{eq:casselsConstantInf}, \eqref{eq:BAtheta}, \eqref{eq:casselsConstant} and \eqref{eq:casselsConstantTrue}. For  $\theta,\;\gamma \in\R$, denote
\begin{equation}\label{eq:FFcasselsConstantInf}
\ctg\myeq\inf_{0\neq N}|N|\cdot|\langle N\theta - \gamma \rangle|
\end{equation}
where $N$ varies over the non zero polynomials in $\Z$,
\begin{equation}\label{eq:FFBAbtheta}
BA_{\theta} \myeq \left\{ \gamma\in\R\sep\ctg>0\right\},
\end{equation}
\begin{equation}\label{eq:FFcasselsConstant}
\ct\myeq\sup_{\gamma}\inf_{0\neq N}|N|\cdot|\langle N\theta - \gamma \rangle|,
\end{equation}
and
\begin{equation}\label{eq:FFcasselsConstantTrue}
c\myeq\inf_{\theta}c(\theta).
\end{equation}
An analogue of Theorem \ref{thm:Kim} was proved in \cite{KN}:
\begin{thm}\label{thm:KimNakada}
For every $\theta\in\R\setminus\Q$, the set $BA_{\theta}$ has zero measure.
\end{thm}
The measure mentioned here is the natural measure on $\R$, which we will recall in Section~\ref{sec:MeasureAndDim}.
Before we formulate higher dimensional analogues, let us introduce a more general notion of weight which is more natural to this context.
A \textit{generalized weight} is a function
$\bg=\left(g^1,\ldots,g^d\right):\bbn\to\bbn^d$, such that for every $1\leq s\leq d$ the function $g^s:\bbn\to\bbn$ is non decreasing, and
\[
\sum_{s=1}^{d}g^s\left(h\right) = h,
\]
for every $h\in\bbn$. Define the higher dimensional versions of \eqref{eq:FFcasselsConstantInf}, \eqref{eq:FFBAbtheta}, \eqref{eq:FFcasselsConstant}, and \eqref{eq:FFcasselsConstantTrue}: For any $\btheta,\;\bgamma \in\R^d$, let
\[
\chtg\myeq\inf_{0\neq N}\max_{1\leq s\leq d}|N|^{\frac{{g^s\left(\deg N\right)}}{\deg N}}\cdot|\langle N\theta^s - \gamma^s \rangle|,
\]
where, by convention, $\frac{g^s\left(0\right)}{0}=1$,
\[
BA_{\btheta}\left(\bg\right) \myeq \left\{ \bgamma\in\R^d\sep\chtg>0\right\},
\]
\[
\cht\myeq\sup_{\bgamma}\chtg ,
\]
and
\[
\ch\myeq\inf_{\btheta}\cht.
\]
While the approach of \cite{S12} is likely to give a proof for the function fields analogue of Theorem~\ref{thm:Shapira}, this line will not be pursued in this note. The reader is referred to \cite{Gho,HP} to learn more about the dynamical approach towards Diophantine approximation in function fields.
\begin{remark}\label{rem:genWeight}
Any weight $\br$ in the sense of Section \ref{subsec:higherDim}, induces a generalized weight $\bg_\br$, by letting $\bg_{\br}(0)\myeq(0,\ldots,0)$ and
$\bg_{\br}(h+1)\myeq\bg_{\br}(h)+\be_s,$ where $1\leq s\leq d$ is any index satisfying
\begin{equation}\label{eq:inducedWeight}
r^s\cdot\left(h+1\right) - g_{\br}^s(h) = \max_{1\leq t\leq d} r^t\cdot\left(h+1\right) - g_{\br}^t(h)
\end{equation}
Note that for every $h\geq0$, we have
\begin{equation}\label{eq:l1norm}
\sum_{1\leq s\leq d} r^s h = \sum_{1\leq s\leq d} g_{\br}^s(h) = h.
\end{equation}
Therefore, there exists $1\leq s\leq d$ such that
\begin{equation}\label{eq:lowerBound}
r^s\cdot\left(h+1\right) - g_{\br}^s(h)\geq \frac{1}{d}.
\end{equation}
By induction on $h$, using \eqref{eq:inducedWeight} and \eqref{eq:lowerBound}, we conclude that
\[
    r^s h - g_{\br}^s(h)\geq - \left(1-\frac{1}{d}\right),
\]
for every $h\geq 0$ and $1\leq s\leq d$. On the other hand, by \eqref{eq:l1norm} and \eqref{eq:lowerBound} we get
\[
	r^s h - g_{\br}^s(h) = - \left(\sum_{t\neq s} r^t h - g_\br^t(h) \right) \leq (d-1)\left(1-\dfrac{1}{d}\right),
\]
The upshot is that for any $\btheta,\bgamma\in\R^d$, the approximation constant $c_{\bg_{\br}}\left(\btheta,\bgamma\right)$ differs from
\[
c_{\br}\left(\btheta,\bgamma\right)\myeq\inf_{0\neq N}\max_{1\leq s\leq d}|N|^{r^s}\cdot|\langle N\theta^s - \gamma^s \rangle|.
\]
by a multiplicative factor smaller than $q^d$. In particular, for every $\btheta\in\R^d$, the set
\[
BA_{\btheta}\left(\br\right) \myeq \left\{ \bgamma\in\R^d\sep c_{\br}\left(\btheta,\bgamma\right) > 0 \right\},
\]
equals $BA_{\btheta}\left(\bg_{\br}\right)$.
\end{remark}

\subsection{Main Results}
In this paper, we prove the function fields analogue of Theorem~\ref{thm:PM} and determine the value of the function fields analogue of \eqref{eq:FFcasselsConstantTrue}. More precisely, we show:
\begin{thm}\label{thm:TsengAnalog}
$BA_{\btheta}(\bg)\neq\varnothing$ for every generalized weight $\bg$ and $\btheta\in\R^d$.
Moreover, if
\begin{equation}\label{eq:moreover1}
\inf_{h\in\bbn} \frac{\min\bg\left(h\right)}{h} > 0
\end{equation}
then $\dim\left(BA_{\btheta}(\bg)\right)=d$ for every $\btheta\in\R^d$.
\end{thm}

\begin{thm}\label{thm:CasAnalog}
Any generalized weight $\bg$ satisfies $\ch = q^{-2}$.
\end{thm}

\begin{remark}
It should be mentioned that \cite{Ar,Ag} deal with a related question concerning products of linear forms. Assume $F_i(x,y)=a_ix+b_iy$, $i\in\{1,2\}$, are two linear forms with coefficients $a_i,b_i\in\R$. Using the methods of \cite{D,C52}, it was proven that:
\begin{equation}\label{eq:ArmAgg}
\sup_{\gamma,\delta\in\R}\inf_{x,y\in\Z}\left\lvert F_1\left(x+\gamma,y+\delta\right)\right\rvert\cdot\left\lvert F_2\left(x+\gamma,y+\delta\right)\right\rvert = \left|a_1b_2-a_2b_1\right|q^{-2}.
\end{equation}
where the upper bound has been already proved in \cite[p. 519]{Mahler}. Given any $\theta\in\R$, take $F_1=\theta x+y$ and $F_2=x$, and plug them into \eqref{eq:ArmAgg} to obtain:
\begin{equation}\label{eq:ArmAggApproximation}
\sup_{\gamma,\delta\in\R}\inf_{N\in\Z}\left\lvert N+\gamma\right\rvert\cdot \left\lvert\left\langle N\theta+\delta+\gamma\theta\right\rangle\right\rvert = q^{-2}.
\end{equation}
Note that forcing $\gamma=0$ and $N\neq0$ can a priori make the left hand side of \eqref{eq:ArmAggApproximation} bigger or smaller, so one cannot apply \eqref{eq:ArmAggApproximation} directly in order to estimate $c(\theta)$.
\end{remark}
\section{Measure and dimension}\label{sec:MeasureAndDim}
In order to prove $BA_{\btheta}$ has the same Hausdorff dimension as $\R^d$, we will construct subsets of it by nested intersection. In this section we recall a general criterion which gives rise to a lower bound on the Hausdorff dimension of such intersections.

\subsection{Tree-like collections}\label{subsec:treeLike}
Let $X$ be a complete metric space with a metric $\rho$, and let $\mu$ be a Borel measure on $X$. Following the terminology of \cite{KW2}, a collection $\calc$ of compact subsets of $X$ is called \emph{tree-like} if there exists a sequence of collections $\{\calc_m\}_{m=0}^\infty$ such that $\calc=\bigcup_{m=0}^\infty\calc_m$ which satisfy the following conditions:
\begin{enumerate}
	\item $\calc_{0}=\{C_0\}$, with $C_0\sub X$ compact.
	\item $\mu(C)>0$ for any $C\in \calc$.
	\item For any $m\in\bbn$ and $C,C'\in\calc_m$, either $C=C'$ or $\mu\left(C\cap C'\right)=0$.
	\item For any $m\in\bbn$ and $C\in\calc_{m+1}$, there exists $C'\in\calc_m$ such that $C\sub C'$.
	\item For any $m\in\bbn$ and $C'\in\calc_{m}$, there exists $C\in\calc_{m+1}$ such that $C\sub C'$.
\end{enumerate}
Given a tree-like collection $\calc=\bigcup_{m=0}^\infty\calc_m$ we define its limit set to be
$$C_\infty=\bigcap_{m=0}^\infty\bigcup_{C\in\calc_m}C.$$
For each $m\in \bbn$ define
$$\rho_m=\sup_{C\in\calc_m}\rho\left(C\right),$$
where $\rho\left(C\right)=\max_{x,y\in C}\rho(x,y)$, and
$$D_m=\inf_{C'\in\calc_m}\frac{\mu\left(\bigcup_{C\in\calc_{m+1},\;C\sub C'}C\right)}{\mu(C')}.$$
A tree like collection is said to be \emph{strongly tree-like} if, in addition:
\begin{enumerate}
\item[6. ]$\rho_m\mto 0$.
\end{enumerate}
The following is a specific case of \cite[Lemma~2.5]{KW2}:
\begin{thm}\label{thm:treeLike}
Let $X$ be a complete metric space with a metric $\rho$, and $\mu$ be a Borel measure. Assume that there exist constants $c,\alpha>0$ such that
\begin{equation}\label{eq:powerLaw}
\mu(B(x,r))\geq cr^\alpha,
\end{equation}
for any $x\in X$ and $0<r<1$. Then any strongly tree-like collection $\calc=\bigcup_{m=0}^\infty\calc_m$ satisfies
$$\dim{C_\infty} \geq \alpha - \limsup_{m\to\infty}\frac {\sum_{k=0}^m\log D_k}{\log \rho_m}.$$
\end{thm}

\subsection{A metric and a measure on $\R^d$}\label{subsec:metricAndMeasure}
We shall make use of the standard metric and measure on $\R$, which will be denoted by $\rho$ and $\mu$ respectively. The metric $\rho$ is defined by $\rho(\theta,\varphi)=|\theta-\varphi|$, for all $\theta,\varphi\in\R$, where $|\cdot|$ stands for the absolute value on $\R$, as described in Section \ref{subsec:functionFieldAnalog}. Note that the balls of this metric are of the form
\[\label{eq:ball}
B\left(\theta,q^{-\ell}\right)=\theta + t^{-\ell}I,
\]
for $\ell\in\bbz$ and $\theta\in\R$. The measure $\mu$ is the Haar measure on $\R$, normalized by $\mu(I)=1$. This measure is characterized by assigning a measure $q^{-\ell}$ to any ball of radius $q^{-\ell}$, and by being invariant under addition.

The metric and the measure on $\R^d$ are defined by
$$\rho^d(\btheta,\bphi)=\max_{1\leq s\leq d}\rho(\theta^s,\varphi^s),$$
for all $\btheta,\bphi\in\R^d$ and
$$\mu^d=\mu\times\ldots\times\mu,$$
$d$ many times. Note that for any $\ell\geq0$,
$$\mu^d\left(B\left(\btheta,q^{-\ell}\right)\right) = q^{-d\ell},$$
and that whenever $q^{-\ell-1} < r \leq q^{-\ell}$, we have
$$B\left(\btheta,r\right) = B\left(\btheta,q^{-\ell}\right).$$
This proves that $\mu^d$ satisfies (\ref{eq:powerLaw}) with $c=1$ and $\alpha=d$.

\subsection{Cantor constructions in $\R^d$}\label{subsec:CantorConstructions}
In this section we describe a construction of a tree-like collection in $\R^d$, which we refer to as a Cantor construction. To introduce the construction, we need some additional notation; For any vector of non negative integers $\bell=(\ell^1,\ldots,\ell^d)$, denote
$$\overline{\bell}=\sum_{s=1}^d\ell^s,$$
and $\bbf_q^{\bell}=\bbf_q^{\overline{\bell}}$. Let $\pi_{\bell}:\R^d\to\bbf_q^{\bell}$ be the projection defined by
\[
\pi_{\bell}(\btheta)\myeq\left(\theta_1^1,\ldots,\theta_{\ell^1}^1,\ldots,\theta_1^d,\ldots,\theta_{\ell^d}^d\right)^t.
\]
For convenience, we denote $\bbf_q^0=\{\varnothing\}$ and $\pi_{(0,\ldots,0)}(\btheta)=\varnothing$ for any $\btheta\in\R^d$. By abuse of notation, let us use $\pi_{\bell}$ to denote the projection to the first $\bell$ coordinates from $\bbf_q^{\bell'}$ to $\bbf_q^{\bell}$, whenever $\bell'\geq\bell$, where this inequality should be understood
coordinatewise. For any $\bv\in\bbf_q^{\bell}$ define the cylinder of $\bv$ by
$$[\bv]\myeq\left\{\btheta\in I^d\sep\pi_{\bell}(\btheta) = \bv \right\},$$
and denote $\bell=\bell(\bv)$. For any $\bh\in\bbn^d$, denote $q^{-\bh}=\left(q^{-h^1},\ldots,q^{-h^d}\right)$. Given a collection of cylinders $\calc$, define
$$q^{-\bh}\calc\myeq\left\{[\bu]\sep[\pi_{\bell}(\bu)]\in\calc,\;\bell=\bell(\bu)-\bh\right\}.$$

Assume $\left(\bell_m\right)_{m=0}^\infty$ is a sequence of $d$ dimensional non negative integer vectors.
Let $\left(\ell_m'\right)_{m=0}^\infty$ be any sequence of non negative integers satisfying $\ell_m'<\overline{\bell_m}$ for all $m$. \sloppy Define a $\left(\left(q^{\bell_m}\right)_{m=0}^\infty,\left(q^{\ell_m'}\right)_{m=0}^\infty\right)$ \textit{Cantor construction} as a set $\left\{\calc_m\sep m\geq0\right\}$ satisfying $\calc_0 = \{I^d\}$, 
$$\calc_{m+1} \sub q^{-\bell_m}\calc_m,$$
and
$$ \left\lvert q^{-\bell_m}\left\{C\right\}\setminus \calc_{m+1}\right\rvert = q^{\ell_m'},$$
for every $m\geq0$ and $C \in \calc_{m}$.
The \textit{limit set} of such a construction is the set
$C_{\infty} = \bigcap_{m=0}^\infty\bigcup_{C\in\calc_m}C $, which we call a $\left(\left(q^{\bell_m}\right)_{m=0}^\infty,\left(q^{\ell_m'}\right)_{m=0}^\infty\right)$ \textit{Cantor set}. If the sequences $\left(\bell_m\right)_{m=0}^\infty,\;\left(\ell_m'\right)_{m=0}^\infty$ are constant, and equal, say, to $\bell, \ell'$ respectively, then we shall call such a set a $\left(q^{\bell},q^{\ell'}\right)$ Cantor set.

\subsection{Measure and dimension of Cantor constructions}\label{subsec:cantorConstruction}
First note that for any $\left(\left(q^{\bell_m}\right)_{m=0}^\infty,\left(q^{\ell_m'}\right)_{m=0}^\infty\right)$ Cantor construction $\left\{\calc_m\sep m\geq0\right\}$, for any $m\geq0$, we have
\[
\mu\left(\bigcup_{C\in\calc_{m+1}}C\right)=\frac {q^{\overline{\bell_m}}-q^{\ell_m'}}{q^{\overline{\bell_m}}}\mu\left(\bigcup_{C\in\calc_{m}}C\right).
\]
This follows from the fact that $\calc_{m+1}$ is composed of equal length cylinders which, therefore, have the same measure. This provides an expression for $\mu(C_\infty)$, and shows that if $\lvert\overline{\bell_m}-\ell_m'\rvert$ is bounded then $$\mu(C_\infty)=0.$$

We now apply Theorem \ref{thm:treeLike} to get a lower bound on the dimension of Cantor sets:
\begin{thm}\label{thm:CantorDim}
Assume $C_\infty$ is a $\left(\left(q^{\bell_m}\right)_{m=0}^\infty,\left(q^{\ell_m'}\right)_{m=0}^\infty\right)$ Cantor set. If
\begin{equation}\label{eq:stronglyTreeLike}
\min \sum_{k=0}^{m-1}\bell_k\mto \infty,
\end{equation}
then
\[\label{eq:HausdorffDim}
\dim (C_\infty)\geq d-\limsup_{m\to\infty}\frac {m+1}{\min \sum_{k=0}^{m-1}\bell_k}\frac{ \log {\frac {q}{q-1}}}{\log q}.
\]
\end{thm}
\begin{proof}
Let $\{\calc_m\sep m\geq0\}$ be the Cantor construction corresponding to $C_\infty$. So $\calc=\bigcup_{m=0}^\infty \calc_m$ is a tree-like collection. Moreover, we have that for every $m\geq0$,
$$\rho_m=q^{-\min \sum_{k=0}^{m-1}\bell_k},$$
and
$$D_m=\frac{q^{\overline{\bell_m}}-q^{\ell_m'}}{q^{{\overline{\bell_m}}}}= 1-q^{\ell_m'-\overline{\bell_m}} \geq \frac{q-1}{q}.$$
(\ref{eq:stronglyTreeLike}) implies that $\calc$ is strongly tree-like. By Theorem \ref{thm:treeLike}, we get
\begin{align*}
 \dim(C_\infty)
  & \geq  d - \limsup_{m\to\infty}\frac{\sum_{k=0}^m\log D_k}{\log \rho_m} \\
  & \geq  d - \limsup_{m\to\infty}\frac {m+1}{\min \sum_{k=0}^{m-1}\bell_k}\frac{ \log {\frac {q}{q-1}}}{\log q}.
\end{align*}
\end{proof}
\section{The One Dimensional Case}\label{sec:oneDim}

In this section we state and prove the one dimensional versions of Theorems \ref{thm:TsengAnalog} and \ref{thm:CasAnalog}. Our method of proof is inspired by \cite{DL}, and utilizes a characterization of the approximations of $\theta$ by means of solutions to a certain linear system of equations.

\subsection{The Corresponding Matrix of an Element in $\R$}\label{sec:correspondingMatrix}
Assume $\theta\in\R$ is a Laurent series, and $N=n_ht^h+...+n_0\in\Z$ is a polynomial of degree $h$. Then
\[
\langle N\theta \rangle =L_1(\theta)t^{-1}+L_2(\theta)t^{-2}+\cdots
\]
where for any $i\geq1$,
\[
L_i(\theta)=n_0\theta_i+n_1\theta_{i+1}+\cdots +n_h\theta_{i+h}.
\]

For any $\gamma\in\R$ and $\ell\geq0$, one has
\begin{equation}\label{eq:matixConditionInitial}
\begin{split}
\left|N\right|\cdot\left|\left\langle N\theta-\gamma\right\rangle \right|  <q^{-(1+\ell)} & \iff
\left|\left\langle N\theta-\gamma\right\rangle \right|<q^{-(h+1+\ell)} \\
&\iff L_i(\theta)=\gamma_i,\quad 1\leq i\leq h+1+\ell.
\end{split}
\end{equation}
In order to write the above linear system of equations in a matrix form, let us define $\Delta(\theta)$ to be the infinite matrix:
\[
\Delta(\theta)=\left(\begin{array}{ccc}
\theta_{1} & \theta_{2} & \cdots\\
\theta_{2} & \theta_{3} & \cdots\\
\vdots &  & \ddots
\end{array}\right).
\]
Denote the $i\times j$ sub-matrix of $\Delta(\theta)$:
\[
\Delta\left[i,j\right]=\left(\begin{array}{ccc}
\theta_{1} & \cdots & \theta_{j}\\
\vdots &  & \vdots\\
\theta_{i} & \cdots & \theta_{i-1+j}
\end{array}\right).
\]
In these notation, we may rewrite \eqref{eq:matixConditionInitial} as
\begin{equation}\label{eq:matrixCondition}
\left|N\right|\cdot\left|\left\langle N\theta-\gamma\right\rangle \right| < q^{-(1+\ell)}  \iff
\Delta\left[h+1+\ell,h+1\right]\cdot \bn =\pi_{h+1+\ell}\left(\gamma\right).
\end{equation}
Here $\bn$ is the coefficients vector of the polynomial $N$.

Consider the same matrix equation, where $\mathbf{n}$ is now a vector of variables. Note that the matrix $\Delta\left[h+1+\ell,h+1\right]$ is a $(h+1+\ell)\times(h+1)$ matrix. Therefore, for any $\ell > 0$ and any fixed $h$, there exists a $\gamma\in\R$ such that equation (\ref{eq:matrixCondition}) has no solutions $\bn\in\bbf_{q}^{h+1}$. This means that for any $N$ of degree $h$, $\left|N\right|\cdot\left|\left\langle N\theta-\gamma\right\rangle \right| \geq q^{-(1+\ell)}$. Our intent is to construct elements $\gamma\in \R$ for which $\ctg \geq q^{-(1+\ell)}$. This is equivalent to the equality on the right hand side of \eqref{eq:matrixCondition} to have no solutions for all $h\geq 0$ at once. To this end, we carefully analyze the rank of the non square submatrices $\Delta[i,j]$.

\begin{remark}\label{rem:Ainfty}
We mention that for $\theta$'s such that $\Delta[m,m]$ is invertible for all $m>0$, our construction is reduced to a slightly easier one. However, it should be noted that the set of $\theta$ for which this happens is a set of measure zero. Indeed, for $\theta\in \R$ and $m>0$,
\begin{equation*}
\det\left(\Delta[m+1,m+1]\right)=\det\left(\Delta[m,m]\right)\theta_{2m+1}+F(\theta_1,\cdots,\theta_{2m}),
\end{equation*}
where $F(\theta_1,\cdots,\theta_{2m})$ is an explicit polynomial which only involves $\theta_1,\cdots,\theta_{2m}$ of $\theta$ (and not $\theta_{2m+1}$).
Therefore, if $\det\left(\Delta[m,m]\right)\neq0$ for all $m$, then
$$\det\left(\Delta\left[m+1,m+1\right]\right)\neq0\;\iff\;\theta_{2m+1}\neq-\frac{F\left(\theta_1,\ldots,\theta_{2m}\right)}{\det\left(\Delta\left[m,m\right]\right)}.$$
Hence, the set of $\theta$'s for which $\Delta[m,m]$ is invertible for all $m>0$ is a $\left(q^2,q\right)$ Cantor set. As discussed in Section~\ref{subsec:cantorConstruction}, such sets have measure zero.
\end{remark}
\subsection{Indices Construction}\label{subsec:IndicesCons}
Given any $\theta\in\R$ and an integer $\ell>0$, we define the sequences of indices $\mathcal{I}_{\ell}=\{i_m\}_{m=0}^\infty,\, \mathcal{J}_{\ell}=\{j_m\}_{m=0}^\infty$ as follows:

\begin{enumerate}
	\item $j_{0}=0$, $i_{0}=\ell$.
	\item $j_{m+1}=\min \{j\sep\rank(\Delta[i_m,j])=i_m\}.$\\
		If this minimum is not obtained, we let $j_{m+1}=\infty$.
	\item If $j_{m+1}=\infty$, let $i_{m+1}=i_m$. Otherwise, define
		\[
		i_{m+1}=\min \{i\sep\rank(\Delta[i,j_{m+1}])=i-\ell\}.
		\]
\end{enumerate}
For convenience, we write $i_{-1}=0$. Note that if $\det(\Delta\left[m,m\right])\neq0$ for all $m>0$, then $i_m=(m+1)\ell$ and $j_m=m\ell$ for all $m\geq0$.
The following lemma summarizes some properties of these indices.

\begin{lemma}\label{lem:indConstruction}
Let $\mathcal{I}_{\ell},\mathcal{J}_{\ell}$ be as defined above. If $j_{m+1}<\infty$, then $i_{m+1}$ is defined, and the indices satisfy:
\begin{enumerate}
\item $i_{m+1} \leq j_{m+1}+\ell\label{eq:imjm}.$
\item $i_{m+1} \geq  i_{m}+\ell\label{eq:im}.$
\item $j_{m+1} \geq j_{m}+\ell\label{eq:jm}.$
\end{enumerate}
\end{lemma}
\begin{proof}
General facts about rank of matrices imply that
\begin{equation}\label{eq:rankUpperBound}
\rank(\Delta[i,j])\leq \min(i,j),
\end{equation}
and
\begin{equation}\label{eq:rankIncreasing}
\rank(\Delta[i,j])\leq\rank(\Delta[i + 1,j]),
\end{equation}
for every $i,j>0$. By the definition of $j_{m+1}$, one has that $i_m - \rank(\Delta[i_m,j_{m+1}]) = 0$. On the other hand, putting $i=j_{m+1} + \ell$ and $j=j_{m+1}$ in \eqref{eq:rankUpperBound} gives $\left(j_{m+1} + \ell\right) - \rank(\Delta[j_{m+1} + \ell,j_{m+1}]) \geq \ell$. By \eqref{eq:rankIncreasing}, any $i,j>0$ satisfy $(i+1) - \rank(\Delta[i+1,j])\leq i - \rank(\Delta[i,j]) + 1$.
Therefore, there exists some $i_m \leq i \leq j_{m+1} + \ell$ for which $i - \rank(\Delta[i,j_{m+1}]) = \ell$. We conclude that $i_{m+1}$ is well defined.
\begin{enumerate}
	\item By the definition of $i_{m+1}$ and the above discussion, it satisfies $i_{m+1}\leq j_{m+1}+\ell$.

    \item Since $\rank(\Delta[i_m,j_{m+1}])=i_m$, we have $i_m\leq j_{m+1}$. It follows that $\rank(\Delta[i,j_{m+1}])=i$ for any $i\leq i_m$, while if $i > i_m$, one has $\rank(\Delta[i,j_{m+1}])\geq i_m$. Using $\rank(\Delta[i_{m+1},j_{m+1}])=i_{m+1}-\ell$, we conclude that $i_{m+1}\geq i_m+\ell$.
	
	\item Note that $\rank(\Delta[i_m,j_m])=i_m-\ell$, and that for all $j\leq j_m$, one has that $\rank(\Delta[i_m,j])\leq i_m-\ell$. By the definition of $j_{m+1}$ as the minimal $j$ for which $\rank(\Delta[i_m,j])=i_m$, it follows that $j_{m+1}\geq j_m+\ell$.
\end{enumerate}	
\end{proof}
\begin{remark}
If $\theta\in\R$ is rational, i.e. $\theta\in\Q$, then there exists an $m$ for which $j_m=\infty$. Indeed, since $\theta$ is rational, its coefficients sequence is eventually periodic, i.e., there exist $m_0,p\in\bbn$ such that $\theta_{m}=\theta_{m+p}$ for all $m\geq m_0$. Therefore, whenever $j_m\geq m_0+p$, we must already have $j_m=\infty$. The implication holds in the other direction as well. Assume that $j_{m+1}=\infty$ for some $m\in\bbn$. Then there exists $0 \neq \bb\in\bbf_q^{i_m}$ such that $\bb^t\cdot (\Delta[i_m,j])=0$ for all $j>0$ at once. Since the $i$-th row of $\Delta(\theta)$ consists of the coefficients of $\left\langle t^{i-1}\theta\right\rangle$, this means that $\sum_{s=1}^{i_m}b^s t^{s-1}\theta$ is a polynomial. Therefore $\theta$ is rational (this argument appears in \cite[p.438]{H}).
\end{remark}
\subsection{Main Proposition}\label{subsec:mainProp}
The following proposition is the key ingredient of the proofs of our main results. To prove it, we make use of the indices constructed in Section~\ref{subsec:IndicesCons}. In fact, the construction of the indices serves as a way to bypass the fact that the matrices $\Delta[m,m]$ are not necessarily invertible.
\begin{prop}\label{prop:mainProp}For any $\theta\in\R$, $\ell>0$ consider the indices sequences $\mathcal{I}_{\ell},\mathcal{J}_{\ell}$ constructed in Section \ref{subsec:IndicesCons}. Let
$\Gamma_{\ell}$ be the set of $\gamma\in\R$ such that for any $m\geq 0$
and $0<j<j_{m+1}$, the equation
\begin{equation}\label{eq:gamma-1}
\Delta\left[i_{m},j\right]\cdot\mathbf{n}=\pi_{i_{m}}\left(\gamma\right)
\end{equation}
has no solutions. Then
\[
\dim\Gamma_{{\ell}}\geq 1 - \frac{1}{\ell}\frac{ \log {\frac {q}{q-1}}}{\log q}.
\]
\end{prop}
\begin{proof}
Let $\calc_0=\left\{ I \right\}$. For $m\geq0$, assume that $\calc_m$ is already defined.
By definition, $\rank(\Delta[i_m,j])\leq i_m-1$ for every $j<j_{m+1}$. Moreover, for every $j\leq j'$ we have
$$\left\{\bb\in\bbf_q^{i_m} \sep \bb^t \cdot \Delta[i_m,j']= \bzero^t\right\}\subseteq \left\{\bb\in\bbf_q^{i_m} \sep \bb^t \cdot \Delta[i_m,j]=\bzero^t\right\}.$$
Hence, there exists $\bzero \neq\bb_m\in\bbf_q^{i_m}$ such that
$$(\bb_m)^t\cdot \Delta[i_m,j]=\bzero^t,$$
for all $j<j_{m+1}$. Define:
\begin{equation*}
\mathcal{C}_{m+1}=\bigcup_{C\in\mathcal{C}_{m}}\left\{\pi_{i_{m}}^{-1}(v)\colon v\in\pi_{i_{m}}(C),\,(\bb_m)^t\cdot v\neq 0\right\}.
\end{equation*}
Note that $\mathcal{C}_{m+1}$ is a set of sets. Finally, define
\[
C_{\infty}=\bigcap_{m=0}^{\infty}\bigcup_{C\in\calc_m}C.
\]	
\textit{Claim 1: }
	$C_\infty\subseteq\Gamma_{\ell}$\\
Let $\gamma\in C_{\infty}$. For $m\geq 0$ and $0<j<j_{m+1}$ we have that
$$(\bb_m)^t\cdot \Delta[i_m,j]=\bzero^t \mbox{ and } (\bb_m)^t\cdot \pi_{i_m}(\gamma)\neq0.$$
Therefore, there are no solutions to \eqref{eq:gamma-1}, and hence $\gamma \in \Gamma_{\ell}$.\\ \\	
\textit{Claim 2: }
	$\dim \left(C_{\infty}\right)\geq 1 - \frac{1}{\ell}\frac{ \log {\frac {q}{q-1}}}{\log q}$ \\
If $j_{m+1}\neq \infty$ for all $m\geq0$, then $C_{\infty}$ is a $\left(\left(q^{i_{m}-i_{m-1}}\right)_{m=0}^\infty, \left(q^{i_{m}-i_{m-1}-1}\right)_{m=0}^{\infty}\right)$ Cantor set. Indeed, for every $m\geq0$, recall that $\rank(\Delta[i_{m-1},j_{m}])=i_{m-1}$, and hence, at least one of the last $i_{m} - i_{m-1}$ coefficients of $\bb_{m}$ is non zero. Therefore, for every $\gamma\in\R$ for which $\pi_{i_{m-1}}^{-1}(\pi_{i_{m-1}}(\gamma))\in\mathcal{C}_m$, there are exactly $q^{i_{m}-i_{m-1}-1}$ vectors $u\in\bbf_{q}^{i_{m}-i_{m-1}}$
for which
\begin{equation*}
\left(\bb_{m}\right)^{t}\cdot
\left(\begin{array}{c}
\pi_{i_{m-1}}\left(\gamma\right)\\
u
\end{array}\right)
=0.
\end{equation*}

Applying Lemma~\ref{lem:indConstruction}(\ref{eq:im}) $m-1$ times, yields $i_{m-1}\geq m\ell $. Since $\sum_{k=0}^{m-1} i_{k} - i_{k-1} = i_{m-1}$, it follows by Theorem~\ref{thm:CantorDim} that
\[
\dim \left(C_{\infty}\right) \geq 1 - \limsup_{m\to\infty}\frac{m+1}{i_{m-1}}\frac{ \log {\frac {q}{q-1}}}{\log q} \geq 1 - \frac {1}{\ell}\frac{ \log {\frac {q}{q-1}}}{\log q}.
\]
If there exists $m\geq0$ for which $j_{m+1}=\infty$, then $C_{\infty}$ is a non empty union of cylinders of length $i_m$, and therefore has a positive measure, thus Hausdorff dimension one.
\end{proof}


\subsection{The One Dimensional Case - Results}\label{sec:proofs}

This section is devoted for the statements and proofs of Theorems \ref{thm:TsengAnalog} and \ref{thm:CasAnalog} in the one dimensional case.
\begin{thm}\label{thm:TsengAnalog1}
For every $\theta\in\R$, $\dim\left(BA_{\theta}\right)=1$.
\end{thm}

\begin{proof}
Fix any $\ell>0$. Consider the sequences $\mathcal{I}_{\ell},\mathcal{J}_{\ell}$ of indices from Section \ref{subsec:IndicesCons}, and the set $\Gamma_{\ell}$ from Proposition~\ref{prop:mainProp}. Assume $\gamma\in\Gamma_{\ell}$. By Proposition~\ref{prop:mainProp}, for all $m \geq 0$ and $0<j<j_{m+1}$, there are no non zero solutions to \eqref{eq:gamma-1}. For any $h\in\bbn$, let $m$ be such that $j_{m}\leq h+1<j_{m+1}$. In particular,
\begin{equation}\label{eq:prev} 
\Delta\left[i_m,h+1\right]\cdot\mathbf{n}=\pi_{i_m}(\gamma),
\end{equation}
has no non zero solutions. Using Lemma~\ref{lem:indConstruction}(\ref{eq:imjm}), we get $i_m\leq j_{m}+\ell \leq h+1+\ell$. Therefore, the equation
\[
\Delta\left[h+1+\ell,h+1\right]\cdot\mathbf{n}=\pi_{h+1+\ell}(\gamma),
\]
has no non zero solutions, as it is obtained from \eqref{eq:prev} by increasing the number of equations.
By \eqref{eq:matrixCondition}, we get that
\[
\left|N\right|\cdot\left|\left\langle N\theta-\gamma\right\rangle \right| \geq q^{-(1+\ell)}
\]
for any $0\neq N\in\Z$. Therefore, $\Gamma_{\ell}\sub BA_{\theta}$. We apply Proposition~\ref{prop:mainProp} to bound the dimension of $BA_{\theta}$ from below:
\[
\dim BA_{\theta}\geq \dim \left(\Gamma_{\ell}\right)\geq 1 - \frac{1}{\ell}\frac{ \log {\frac {q}{q-1}}}{\log q}.
\]
Since the above holds for all $\ell>0$, and since $1-\frac{1}{\ell}\frac{ \log {\frac {q}{q-1}}}{\log q}\lto 1$, we conclude that $\dim\left(BA_{\theta}\right)=1$.
\end{proof}

\begin{prop}\label{prop:CasLowerBound1}	
For every $\theta\in\R$ one has that $\ct\geq q^{-2}$.
\end{prop}
\begin{proof}
The proof of Theorem~\ref{thm:TsengAnalog1} shows that any $\gamma\in\Gamma_1$ satisfies $\ctg\geq q^{-2}$. Proposition~\ref{prop:mainProp} implies $\dim\left(\Gamma_1\right) > 0$, hence in particular $\Gamma_1\neq\varnothing$. Therefore, $\ct\geq q^{-2}$.
\end{proof}

We now give a property of the elements $\theta\in\R$ for which $c(\theta)\geq q^{-1}$.
\begin{prop}\label{prop:CasUpperBound1}
If $c(\theta)\geq q^{-1}$ then there exists $m_0\in\bbn\cup\{\infty\}$ such that
\begin{equation}\label{eq:m0}
\Delta[m,m] \text{ is invertible exactly for }  0 < m < m_0.
\end{equation}
\end{prop}
\begin{proof}
Assume that there is no $m_0$ satisfying \eqref{eq:m0}. Therefore, there are $0< m_1 < m_2 <\infty$ such that $\Delta[m_1,m_1]$ is not invertible and $\Delta[m_2,m_2]$ is invertible. By assumption, there exists $\gamma$ such that
\begin{equation}\label{eq:solutions}
\left|N\right|\cdot\left|\left\langle N\theta-\gamma\right\rangle \right| < q^{-1}
\end{equation}
has no solutions $0\neq N$ in $\Z$. By \eqref{eq:matrixCondition} with $\ell=0$,
\begin{equation}\label{eq:basicEquation}
\Delta\left[m,m\right]\cdot \bn=\pi_m\left(\gamma\right)
\end{equation}
has no solutions $0\neq\bn\in \bbf_q^m$ for any $m>0$ with $n_m\neq0$. Therefore, there are no non zero solutions to \eqref{eq:basicEquation}. In particular, $\Delta[m_2,m_2]\cdot\bn=\pi_{m_2}(\gamma)$ has no non zero solutions $\bn\in\bbf_q^{m_2}$. However, $\Delta[m_2,m_2]$ is invertible, so we must have $\pi_{m_2}(\gamma)=\bzero$. Since $m_1 < m_2$, we have $\pi_{m_1}(\gamma)=\bzero$. Now, $\Delta[m_1,m_1]$ is non-invertible, therefore, the equation $\Delta[m_1,m_1]\cdot\pi_{m_1}\left(\bn\right)=\bzero$ has non zero solutions, contradicting \eqref{eq:basicEquation} for $m_1$.
\end{proof}

\begin{remark}
In the extreme cases $m_0=1$ and $m_0=\infty$, the other implication also holds. To see this, note that if $\det\left(\Delta[m,m]\right)=0$ for all $m>0$ then we must have $\theta=0$. Indeed, for any $m>0$, assume that $\theta_1=\ldots=\theta_{m-1}=0$. Then $\Delta[m,m]$ have $\theta_{m}$ on the anti diagonal, and zeroes above it. Therefore, $0=\det\left(\Delta[m,m]\right)=\left(\theta_{m}\right)^m$, so $\theta_m=0$. For $\theta=0$, any $\gamma\neq0$ does not have solutions for \eqref{eq:basicEquation}. If $m_0=\infty$, choose $\gamma=0$. Since $\Delta[m,m]$ is invertible for every $m$, the only solution to \eqref{eq:basicEquation} is $\bn=\bzero$.
\end{remark}
As a corollary of Propositions \ref{prop:CasLowerBound1} and \ref{prop:CasUpperBound1}, we get:
\begin{thm}\label{thm:CasAnalog1}	
$c = q^{-2}$.
\end{thm}
\begin{proof}
One only needs to make sure that there exists $\theta$ such that $c(\theta)=q^{-2}$. It is enough to find $\theta$ which does not satisfy the conclusion of Proposition~\ref{prop:CasUpperBound1}. Any $\theta$ with $\theta_1=0$ and $\theta_2\neq0$ works.
\end{proof}

We complete the discussion on the one dimensional case by showing that replacing the $\inf$ by $\liminf$ in the definition of $\ctg$ does not change the value of the constant:

\begin{prop}\label{prop:CasUpperBoundLiminf}
Let $\theta\in\R$. If $\tilde{c}\left(\theta\right)\myeq\sup_{\gamma}\liminf \left\{|N||\langle N\theta-\gamma\rangle|\sep 0\neq N\in \Z\right\} \geq q^{-1}$ then there exists $m_0\in\bbn\cup\{\infty\}$ such that
$\Delta[m,m]$ is either invertible for all  $m \geq m_0$ or non invertible for all $m \geq m_0$.
\end{prop}

\begin{proof}
The proof here is similar to the proof of Theorem~\ref{prop:CasUpperBound1}. Assume that there are infinitely many pairs $0 < m_1 < m_2$ for which $\Delta\left[m_1,m_1\right]$ is non invertible and $\Delta\left[m_2,m_2\right]$ is invertible. This implies that \eqref{eq:solutions} has infinitely many non zero solutions, hence, $\tilde{c}\left(\theta\right) < q^{-1}$, which contradicts the assumptions of the proposition.
\end{proof}

\begin{thm}\label{thm:liminf}
$\tilde{c}\myeq\inf_{\theta}\tilde{c}\left(\theta\right)=q^{-2}.$
\end{thm}
\begin{proof}
By definition we have $c\leq\tilde{c}$, so it is enough to find $\theta$ for which $\tilde{c}(\theta)\leq q^{-2}$. Define $\theta$ by $\theta_{m_k}=1$ for the sequence $m_k=2^{k+1}-2$, $k=1,2,\ldots $, and $\theta_m=0 $ for every other $m\in\bbn\setminus\left\{m_k\sep k\in\bbn\right\}$. For this $\theta$ we have that $\Delta[m_k,m_k]$ is invertible because the anti diagonal is full with ones, and below the anti diagonal there are only zeros. On the other hand, $\Delta[m_k+1,m_k+1]$ is non invertible since the last row and column are zero. Hence, by Proposition \ref{prop:CasUpperBoundLiminf}, $\tilde{c}(\theta)\leq q^{-2}$.
\end{proof}

\section{The General Case}\label{sec:generalCase}

We now turn to prove Theorems \ref{thm:TsengAnalog} and \ref{thm:CasAnalog}. To this end, we need to further
generalize our indices construction. Fix a generalized weight $\bg$, a vector $\btheta\in\R^d$ and $\ell>0$, and define the matrices
\begin{equation}\label{eq:matrixHigherDimension}
\Delta\left[i,j\right]=\left(\begin{array}{ccc}
\theta_{1}^{1} & \cdots & \theta_{j}^{1}\\
\vdots &  & \vdots\\
\theta_{g^{1}\left(i\right)}^{1} & \cdots & \theta_{g^1\left(i\right)+j}^{1}\\
\cdots & \cdots & \cdots\\
\theta_{1}^{d} & \cdots & \theta_{j}^{d}\\
\vdots &  & \vdots\\
\theta_{g^d\left(i\right)}^{d} & \cdots & \theta_{g^d\left(i\right)+j}^{d}
\end{array}\right).
\end{equation}
We construct the set of indices $\cali_{\bg,\ell},\;\calj_{\bg,\ell}$ the same way as in the one dimensional case. This construction has the same properties as summarized in Lemma~\ref{lem:indConstruction}, as one can prove by repeating the proof of Lemma~\ref{lem:indConstruction} verbatim. The same argument works since the rank is independent of the order of the rows.

Define the set
\[
BA_{\btheta}\left(\bg,\ell\right)  =  \left\{ \bgamma\in\R^d\sep\inf_{N\neq0}\max_{1\leq s\leq d}\left|N\right|^{\frac{g^s\left(\deg N+1+\ell\right)}{\deg N}}\left|\left\langle N\theta^{s}-\gamma^{s}\right\rangle \right|\geq1\right\}.
\]
For every $1\leq s\leq d$, we have $g^s(n+1)\leq g^s(n)+1$ for all $n$. It follows that
\begin{equation}\label{eq:BAh}
BA_{\btheta}\left(\bg,\ell\right)\sub BA_{\btheta}\left(\bg\right).
\end{equation}
Note that for any polynomial $N$ of $\deg N=h$, and every $\ell>0$, one has:
\begin{equation}
\begin{array}{ll}
\max_{1\leq s\leq d} \left|N\right|^{\frac{g^s\left(h+1+\ell\right)}{h}}\left|\left\langle N\theta^{s}-\gamma^{s}\right\rangle \right|  < 1   & \iff \\
&\\
\Delta\left[h+1+\ell,h+1\right]\bn   =  \pi_{\bg}\left(\bgamma\right),
\end{array}
\end{equation}

where $\bn$ is the coefficients vector of the polynomial $N$. This is the higher dimensional version of \eqref{eq:matrixCondition}. The next proposition is the higher dimensional version of Proposition \ref{prop:mainProp}, and the idea of the proof is the same. Therefore, we will mainly emphasize the differences in the proof.

\begin{prop}\label{prop:mainProp2}
Assume $\btheta\in\R^{d}$, a generalized weight $\bg$, and $\ell>0$. Define $\Gamma_{\btheta}\left(\bg,\ell\right)$ as the set
of all $\bgamma\in\R^{d}$ such that for any $m\in\bbn$ and $0<j<j_{m+1}$,
the equation
\begin{equation}
\Delta\left[i_{m},j\right]\bn=\pi_{\bg\left(i_{m}\right)}\left(\bgamma\right)
\end{equation}
has no solutions $\bn\in\bbf_{q}^{j}$. Then $\Gamma_{\btheta}\left(\bg,\ell\right)\neq\varnothing$. Moreover, if
\begin{equation}\label{eq:moreover2}
\min \bg\left(i_m\right) \mto \infty
\end{equation}
then
\begin{equation}\label{eq:dimGamma}
\dim \left(C_{\infty}\right) \geq d - \limsup_{m\to\infty}\frac{m+1}{\min{\bg(i_{m-1})}}\frac{ \log {\frac {q}{q-1}}}{\log q}.
\end{equation}
\end{prop}

\begin{proof}
Let $\calc_0 = \left\{ I^d \right\}$. In the same way as is in the proof of Proposition \ref{prop:mainProp} define for each $m \geq 1$ the sets $\calc_m$, vectors $\bb_m \in \bbf_q^{i_m}$ and the set $\calc_\infty$, but using the matrices \eqref{eq:matrixHigherDimension}, and projections $\pi_{\bg(i)}$ instead of $\pi_i$.\\ \\
\textit{Claim 1: }
	$C_\infty\subseteq\Gamma_{\btheta}\left(\bg,\ell\right)$.\\
The argument is the same as in Proposition \ref{prop:mainProp}, but with the two aforementioned changes.\\ \\
\textit{Claim 2: }
	If $j_{m+1} \neq \infty$ for all $m$, then $C_\infty$ is a $\left(\left(q^{\bg(i_{m})-\bg(i_{m-1})}\right)_{m=0}^\infty, \left(q^{i_{m}-i_{m-1}-1}\right)_{m=0}^{\infty}\right)$ Cantor set.\\
	An analysis like the one in Proposition \ref{prop:mainProp} gives, that for each $m \geq 0$, and for each $C \in \calc_m$, there are exactly $q^{i_{m}-i_{m-1}-1}$ vectors $v \in \pi_{\bg(i_m)}(C)$, for which $(\bb_m)^t\cdot v = 0$. From the construction of $\calc_{m+1}$ from $\calc_m$, we get the desired.\\ \\
\textit{Claim 3: }
	$\Gamma_{\btheta}\left(\bg,\ell\right)\neq\varnothing$.\\
If $j_{m+1} \neq \infty$ for all $m$, it follows since Cantor sets are non empty, that $C_\infty \neq \varnothing$, and hence $\Gamma_{\btheta}\left(\bg,\ell\right)\neq\varnothing$.

If there exists $m\geq0$ for which $j_{m+1}=\infty$, then $C_{\infty}$ is a non empty union of cylinders of length $\bg(i_m)$. Therefore it is non empty.\\ \\
\textit{Claim 4: }If	
\[
	\min \bg\left(i_m\right) \mto \infty,
\]
then
\[
\dim \left(C_{\infty}\right) \geq d - \limsup_{m\to\infty}\frac{m+1}{\bg(i_{m-1})}\frac{ \log {\frac {q}{q-1}}}{\log q}.
\]

If $j_{m+1} \neq \infty$ for all $m$, then since $\sum_{k=0}^{m-1} \bg(i_{k}) - \bg(i_{k-1}) = \bg(i_{m-1})$, the result follows from Theorem~\ref{thm:CantorDim}.

If there exists $m\geq0$ for which $j_{m+1}=\infty$, then $C_\infty$ has positive measure, and hence dimension $d$.
\end{proof}

\begin{proof}[Proof of Theorem \ref{thm:TsengAnalog}] Recall that we want to show that $\dim\left(BA_{\theta}(\bg)\right)=d$. Let $\ell>0$ be any integer.
By imitating the proof of Theorem \ref{thm:TsengAnalog1}, we get that $\Gamma_{\btheta}\left(\bg,\ell\right)\sub BA_{\btheta}\left(\bg,\ell\right)$.
By Proposition \ref{prop:mainProp2}, we get $\Gamma_{\btheta}\left(\bg,\ell\right)\neq\varnothing$, and hence, $BA_{\btheta}\left(\bg,\ell\right)\neq\varnothing$. To conclude the second part of the theorem, we assume that (\ref{eq:moreover1}) holds. Therefore, there exists $r>0$ such that $\min\bg(h)\geq rh$ for all $h$. By applying Lemma~\ref{lem:indConstruction}(\ref{eq:im}) $m$ times we see that for every $m\geq0$, $i_m\geq (m+1)\ell$. By the monotonicity of $g^s$ for all $1\leq s\leq d$, we thus obtain that
\[
\min \bg\left(i_m\right) \geq \min \bg\left((m+1)\ell\right) \geq r(m+1)\ell \mto \infty.
\]
Hence, condition (\ref{eq:moreover2}) is satisfied. As a consequence of Proposition~\ref{prop:mainProp2}, the inequality (\ref{eq:dimGamma}) also holds.

Finally,
\begin{align*}
\dim BA_{\btheta}\left(\bg,\ell\right)
& \geq \dim\left(\Gamma_{\btheta}\left(\bg,\ell\right)\right) \\
& \geq d - \limsup_{m\to\infty}\frac {m+1}{\min \bg\left(i_{m-1}\right)}\frac{ \log {\frac {q}{q-1}}}{\log q} \\
& \geq d - \frac {1}{r\ell}.
\end{align*}
As $\ell>0$ is arbitrary, by (\ref{eq:BAh}) we get
$$\dim BA_{\btheta}\left(\bg\right)=d.$$
\end{proof}

\begin{proof}[Proof of Theorem \ref{thm:CasAnalog}] We want to show that $\ch=q^{-2}$. As in the proof of Proposition~\ref{prop:CasLowerBound1}, we note that $\Gamma_{\btheta}\left(\bg,1\right)$ is not empty. Therefore, $\cht\geq q^{-2}$ for every $\btheta\in\R^d$. To show equality, it is enough to find one $\btheta$ for which $\cht=q^{-2}$. Let $1\leq s_1,s_2\leq d$ be such that $g^{s_1}(1)\neq0$ and $g^{s_2}(1)\neq g^{s_2}(2)$. If $s_1=s_2$, choose any $\btheta$ with $\theta^{s_1}_1=0$, $\theta^{s_1}_2=1$ and $\theta^{s_1}_3=0$. Otherwise, choose $\theta^{s_1}_1=0$, $\theta^{s_1}_2=1$, $\theta^{s_2}_1=1$ and $\theta^{s_2}_2=0$.
\end{proof}

\begin{appendices}

\subsection*{Acknowledgements}
We would like to thank Lior Bary-Soroker, Simon Kristensen and Barak Weiss for an abundance of remarks and support, and for suggesting that we work together on this project. We are indebt for the anonymous referee for carefully reading this note, helping us to greatly improve it. The time spent by the first and second named authors at ICERM semester program on "Dimension and Dynamics" had a significant role in the publication of this note. Part of this work was funded by ERC starter grant DLGAPS 279893, BSF grant 2010428, and by the Danish Natural Science Research Council.

\end{appendices}

\bibliographystyle{alpha}
\nocite{*}

\end{document}